\documentclass[11pt]{amsart}
\title[Some lower bounds for the Kirby-Thompson invariant]
{Some lower bounds for the Kirby-Thompson invariant}
\author{Nobutaka Asano}
\address{National Institute of Technology, Tsuyama College, 624-1 Numa, Tsuyama-shi, Okayama, 708-8509, Japan}
\email{asano-n@tsuyama-ct.ac.jp}
\author{Hironobu Naoe}
\address{Department of Mathematics, Chuo University, 1-13-27 Kasuga Bunkyo-ku, Tokyo, 112-8551, Japan}
\email{naoe@math.chuo-u.ac.jp}
\author{Masaki Ogawa}
\address{Department of Mathematics, Saitama University, 255 Shimo-Okubo, Sakura-ku, Saitama-shi, Saitama, 338-8570, Japan}
\email{m.ogawa.691@ms.saitama-u.ac.jp}
\date{}

\usepackage{amsfonts,amsmath,amssymb,amscd}
\usepackage{amsthm}
\usepackage{latexsym}
\usepackage[dvipdfmx]{graphicx}
\usepackage{txfonts}
\usepackage{ulem}
\usepackage{color}
\usepackage{bm}

\usepackage{url}
\usepackage{psfrag}
\usepackage{color}

\theoremstyle{plain}
\newtheorem{theorem}{Theorem}[section]
\newtheorem{lemma}[theorem]{Lemma}
\newtheorem{corollary}[theorem]{Corollary}
\newtheorem{proposition}[theorem]{Proposition}

\newtheorem{maintheorem}{Theorem}

\theoremstyle{definition}
\newtheorem{definition}[theorem]{Definition}

\newtheorem{remark}[theorem]{Remark}

\theoremstyle{definition}
\newtheorem{claim}[theorem]{Claim}

\newcommand{\C}{\mathbb{C}}

\newcommand{\Int}{\mathrm{Int}}

\renewcommand{\P}{\mathcal{P}}
\newcommand{\D}{\mathcal{D}}
\newcommand{\T}{\mathcal{T}}
\newcommand{\disjoint}{\varnothing}

\definecolor{myred}{rgb}{.8,.0,.0}
\definecolor{mygreen}{rgb}{.0,.6,.0}
\definecolor{mygray}{gray}{0.7}

\newcounter{mystepcount}
\newenvironment{mystep}
{
\begin{list}{\it Step \arabic{mystepcount}.}
{
\usecounter{mystepcount}
\setlength{\topsep}{8pt}
\setlength{\itemindent}{0pt}
\setlength{\leftmargin}{47pt} %
\setlength{\rightmargin}{0pt} %
\setlength{\labelsep}{11pt} %
\setlength{\labelwidth}{33pt} %
\setlength{\itemsep}{2pt} %
\setlength{\parsep}{0pt} %
\setlength{\listparindent}{11pt} %
}
}{
\end{list}
}

\begin{document}
\begin{abstract}
Kirby and Thompson introduced a non-negative integer-valued invariant, called the Kirby-Thompson invariant, 
of a $4$-manifold using trisections. 
In this paper, we give some lower bounds for the Kirby-Thompson invariant of certain $4$-manifolds. 
As an application, we determine the Kirby-Thompson invariant of the spin of $L(2,1)$, 
which is the first example of a $4$-manifold with non-trivial Kirby-Thompson invariant. 
We also show that there exist $4$-manifolds with arbitrarily large Kirby-Thompson invariant. 
\end{abstract}
\maketitle
\section{Introduction}
\label{sec:Introduction}
A trisection is a decomposition of a compact connected oriented smooth $4$-manifold into three 1-handlebodies.
This is an analogue of a Heegaard splitting of a $3$-manifold.
The notion of a trisection was introduced by Gay and Kirby in \cite{GK}.
They showed that any closed connected oriented smooth $4$-manifold admits a trisection.
One of the features of a trisection is that it can be described by a trisection diagram, 
which is given as a $4$-tuple consisting of a closed surface and three cut systems 
(for more details, see Section~2). 
We can uniquely recover the trisection and the $4$-manifold from its trisection diagram up to stabilization and diffeomorphism, respectively, 
and thus we have obtained a new combinatorial description of $4$-manifolds. 

Kirby and Thompson introduced a non-negative integer-valued invariant of closed connected oriented smooth $4$-manifolds using trisection diagrams in \cite{KT}.
This invariant is called the Kirby-Thompson invariant, and it is denoted by $L_X$ for a $4$-manifold $X$.
Loosely speaking, $L_X$ measures how complicated a trisection diagram of $X$ must be. 
In \cite{KT}, Kirby and Thompson classified $4$-manifolds with $L_X=0$: 
they showed that $L_X=0$ if and only if $X$ is diffeomorphic to the connected sum of some copies of $S^4$, $S^1\times S^3$, $S^2\times S^2$, $\C P^2$ and $\overline{\C P^2}$. 
They also showed in the same paper that there exist no $4$-manifolds with $L_X=1$, 
and then the third author showed that there exist no $4$-manifolds with $L_X=2$ in~\cite{O}.
This invariant was extended to the case where $X$ is a $4$-manifold with boundary in \cite{rel}. 
It was also adapted to knotted surfaces using bridge trisections in \cite{Br}, see also \cite{APTZ,APZ}. 

The Kirby-Thompson invariant of a closed $4$-manifold $X$ is defined as the minimum number of a kind of ``length'' associated to a trisection diagram of $X$ (see Subsection~\ref{subsec:Kirby-Thompson_inv} for the precise definition).
It is difficult to determine the Kirby-Thompson invariant for a given $4$-manifold.
It is natural and important to seek lower bounds for the such invariant, which is just what we will do in this paper. 
Note that already lower bounds for the Kirby-Thompson invariants of $4$-manifolds with boundary and knotted surfaces 
have been studied for instance in \cite{APTZ,APZ,Br,rel}, 
but those results are entirely out of our setting, that is, the original case where $X$ is a closed $4$-manifold. 

To state our results, we recall the definition of {\it geometrically simply connected}. 
A smooth manifold $X$ is said to be geometrically simply connected 
if it admits a handle decomposition without $1$-handles. 
It is obvious that a geometrically simply connected manifold is simply connected. 
The first two theorems of ours are the following. 
\begin{maintheorem}
\label{thm:1}
Let $X$ be a closed connected oriented smooth $4$-manifold. 
If $X$ is not geometrically simply connected and does not contain $S^1\times S^3$ as a connected summand, 
then $L_X\geq4$. 
\end{maintheorem}
\begin{maintheorem}
\label{thm:2}
Let $X$ be a closed connected oriented smooth $4$-manifold. 
Suppose that $X$ does not contain $S^1\times S^3$ as a connected summand. 
If $\pi_1(X)$ is not trivial nor infinite cyclic, 
then $L_X\geq6$. 
\end{maintheorem}

As a consequence of Theorem~\ref{thm:2}, 
we determine the Kirby-Thompson invariant of the spin of $L(2,1)$ in Corollary~\ref{cor:S(L(2,1))}, which is exactly $6$. 
This is the first example of a $4$-manifold of which
Kirby-Thompson invariant is determined in the case $L_X>0$. 

In either of the above two threorems, we make the assumption that $X$ does not contain $S^1\times S^3$ as a connected summand. 
As pointed out in \cite{KT}, the Kirby-Thompson invariant is subadditive on connected sums, that is, 
the inequality $L_{X\# X'}\leq L_{X}+L_{X'}$ holds in general. 
If the equality holds for $X'=S^1\times S^3$, we can remove the assumption on connected sum from the theorems. 
Therefore, it would be interesting to ask whether there exists a closed connected oriented smooth $4$-manifold $X$ such that $L_{X\# (S^1\times S^3)}<L_{X}$. 

While the lower bounds in Theorems~1 and 2 are given as constants, 
we give a non-constant lower bound for $L_X$ in Theorem~3. 
In particular, we show that there exist $4$-manifolds with arbitrarily large Kirby-Thompson invariant.
\begin{maintheorem}
\label{thm:3}
There exists a universal positive constant $C$ such that 
\[
L_X\geq C\sqrt{\,\log |H_1(X)|\,} 
\]
for any closed connected oriented smooth $4$-manifold $X$ with $2\leq|H_1(X)|<\infty$.
\end{maintheorem}
This paper is organized as follows. 
In Section~\ref{sec:Preliminaries}, 
we review the basic notions of a trisection and a trisection diagram. 
We also demonstrate how to draw a Kirby diagram from a trisection diagram, 
and then we recall the definition of the Kirby-Thompson invariant.
We prove Theorem~\ref{thm:1} in Sections~\ref{sec:l_alpha=0and1} and \ref{sec:Thm1}. 
In Section~\ref{sec:Thm2}, we prove Theorem~\ref{thm:2} by using techniques developed in Sections~\ref{sec:l_alpha=0and1} and \ref{sec:Thm1}.
Finally, we show Theorem~\ref{thm:3} in Section~\ref{sec:Thm3}.

The second author is supported by JSPS KAKENHI Grant Number JP20K14316, 
and the third author is supported by JSPS KAKENHI Grant Number JP20J20545. 
\section{Preliminaries}
\label{sec:Preliminaries}
Throughout this paper, we work in the smooth category unless otherwise stated.

We will write the connected sum of $k$ copies of a closed $n$-manifold $X$ as $\#^k X$, 
and also we write the boundary connected sum of $k$ copies of an $n$-manifold $M$ with connected boundary as $\natural^k M$. 
We adopt the conventions $\#^0 X=S^n$ and $\natural^0 M=B^n$. 
For two triangulable spaces $A\subset B$, a regular neighborhood of $A$ in $B$ will be denoted by $N(A;B)$.

For an oriented surface $\Sigma$, 
the algebraic intersection number of two oriented simple closed curves $\gamma$ and $\gamma'$ in $\Sigma$ will be denoted by $\langle\gamma,\gamma'\rangle$. 

\subsection{A trisection of a $4$-manifold}
\label{subsec:Trisection}
We start with the definition of a trisection of a closed connected oriented $4$-manifold. 
Trisections were firstly introduced by Gay and Kirby in \cite{GK}. It is a decomposition of a $4$-manifold into three 1-handlebodies. 

 \begin{definition}
\label{def:trisection}
Let $g,k_1,k_2$ and $k_3$ be non-negative integers with $\max\{k_1,k_2,k_3\}\leq g$. 
A $(g; k_1, k_2, k_3)$-{\it trisection} (or simply a {\it trisection}) of a closed connected oriented $4$-manifold $X$ 
is a decomposition $X=X_1\cup X_2\cup X_3$ such that for $i,j\in\{1,2,3\}$, 
\begin{itemize}
 \item
$X_i\cong \natural^{k_i}(S^1\times B^3)$, 
 \item
$X_i\cap X_j\cong \natural^g (S^1\times B^2)$ if $i\neq j$, and 
 \item
$\Sigma=X_1\cap X_2\cap X_3 \cong \#^g (S^1\times S^1)$. 
\end{itemize}
\end{definition}
In \cite{GK}, Gay and Kirby showed that any closed connected oriented $4$-manifold admits a trisection. 

The surface $\Sigma=X_1\cap X_2\cap X_3$ is canonically equipped with an orientation in the following way. 
The normal bundle of $\Sigma$ in $X$ is a trivial $B^2$-bundle 
since $\Sigma$ is orientable and null-homologous. 
By definition, this bundle $\Sigma\times B^2$ is also decomposed into three pieces $(\Sigma\times B^2)\cap X_1$, $(\Sigma\times B^2)\cap X_2$ and $(\Sigma\times B^2)\cap X_3$. 
Especially, the boundary of a fiber $\{\text{pt.}\}\times B^2$ is decomposed into three arcs going through $X_1, X_2$ and $X_3$, 
and then it induces an orientation of the fiber $B^2$ with the order $X_1, X_2, X_3$. 
Thus, an orientation of $\Sigma$ is also determined from those of $X$ and $B^2$. 
This oriented surface $\Sigma$ is called the {\it trisection surface} of the trisection $X=X_1\cup X_2\cup X_3$. 

We next explain a trisection diagram. 
A {\it cut system} of a genus $g$ closed surface $\Sigma$ is an ordered tuple $\alpha=(\alpha_1,\ldots,\alpha_g)$ of 
mutually disjoint $g$ simple closed curves $\alpha_1,\ldots,\alpha_g$ such that $\Sigma\setminus\bigcup_{i=1}^g\alpha_i$ is a sphere with $2g$ punctures. 
We also define a {\it cut system} of a $3$-dimensional handlebody $H$ of genus $g$ as a cut system of $\partial H$ 
if there exists a properly embedded disk $D_i$ in $H$ with $\partial D_i=\alpha_i$ for each $i\in\{1,\ldots,g\}$ such that $H\setminus \Int\, N(D_1\cup\cdots\cup D_g;H)$ is a $3$-ball.
Note that a cut system for any closed surface always exists, and so does that for any handlebody. 

Let $X_1\cup X_2\cup X_3$ be a $(g; k_1, k_2, k_3)$-trisection of $X$ 
and $\Sigma$ the trisection surface. 
As in Definition~\ref{def:trisection}, the surface $\Sigma$ bounds three $1$-handlebodies $X_1\cap X_2$, $X_2\cap X_3$ and $X_3\cap X_1$. 
Then there exist cut systems $\alpha, \beta$ and $\gamma$ on the trisection surface $\Sigma$ for $X_1\cap X_2$, $X_2\cap X_3$ and $X_3\cap X_1$, respectively. 
They consists of $3g$ circles drawn on $\Sigma$, 
and the $4$-tuple $(\Sigma; \alpha, \beta, \gamma)$ is called a {\it trisection diagram}. 
By the definition of a trisection, $(\Sigma; \alpha, \beta)$, $(\Sigma; \beta, \gamma)$ and $(\Sigma; \gamma, \alpha)$ 
are genus $g$ Heegaard diagrams of $\#^{k_1}(S^1\times S^2)$, $\#^{k_2}(S^1\times S^2)$ and $\#^{k_3}(S^1\times S^2)$, respectively. 

Given a trisection diagram $\T=(\Sigma; \alpha, \beta, \gamma)$, we can recover the corresponding $4$-manifold as follows (see \cite{GK} for the details). 
We first consider the product $\Sigma\times B^2$ of $\Sigma$ and $B^2$, 
where $B^2$ is the unit disk in $\mathbb{C}$. 
Next we attach $g$ $2$-handles along $\alpha\times\{1\}$ with the surface framing of $\Sigma$, 
and then we attach $2g$ $2$-handles along $\beta\times\{e^{2\pi\sqrt{-1}/3}\}$ and $\gamma\times\{e^{4\pi\sqrt{-1}/3}\}$ as well. 
Notice that the boundary of the obtained manifold is diffeomorphic to the connected sum of some copies of $S^1\times S^2$. 
Then we can make a closed $4$-manifold in a canonical way by \cite{LP72}, 
which is what we want to construct. 

\subsection{Kirby diagram obtained from a trisection diagram}
\label{subsec:Kirby_diag}
Here we explain how to draw a Kirby diagram from a trisection diagram. 
We begin with introducing notations that will be used throughout this paper. 
\begin{definition}
Let $c$ and $c'$ be simple closed curves in an orientable surface $\Sigma$. 
If $c$ and $c'$ are isotopic to each other, then we write $c\P c'$. 
If $c$ and $c'$ intersect transversely at exactly one point, then we write $c\D c'$. 
\end{definition}

Let $\T=(\Sigma;\alpha,\beta,\gamma)$ be a $(g;k_1,k_2,k_3)$-trisection diagram for a closed connected oriented $4$-manifold $X$. 
Suppose that there exists a permutation $\sigma$ of $\{1,\ldots,g\}$ 
such that either $\alpha_i\P\beta_{\sigma(i)}$ or $\alpha_i\D\beta_{\sigma(i)}$ holds for each $i\in\{1,\ldots,g\}$. 
Note that any closed connected oriented $4$-manifold admits such a diagram. 
Then we reorder the curves of $\alpha$ and $\beta$ so that 
$\alpha_i\P\beta_i$ for $i\in\{1,\ldots,k_1\}$ and also that 
$\alpha_i\D\beta_i$ for $i\in\{k_1+1,\ldots,g\}$ for simplicity. 
A Kirby diagram for $X$ is obtained from $(\Sigma;\alpha,\beta,\gamma)$ according to the following steps. 
\begin{figure}[t]
\includegraphics[width=80mm, bb = 0 0 220 89]{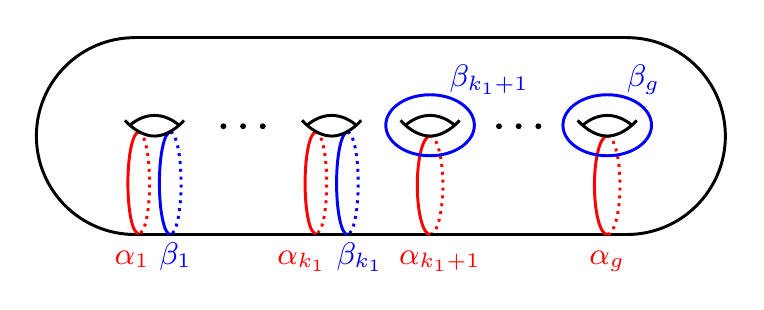}
\caption{The positions of curves of $\alpha$ and $\beta$ in $\Sigma\subset S^3$.}
\label{fig:std_pos}
\end{figure}
\begin{mystep}
\item
We first embed $\Sigma$ in $S^3$ so that the curves of $\alpha$ and $\beta$ lie 
as shown in figure~\ref{fig:std_pos}. 
\item
Let $H_-$ and $H_+$ be two handlebodies in $S^3$ such that $H_-\cap H_+=\Sigma$, $H_-\cup H_+=S^3$ and 
$\alpha$ is a cut system for $H_-$. 
\item
For $i\in\{1,\ldots,k_1\}$, 
let $L^1_{i}$ be a curve obtained from $\alpha_i$ by slightly pushing it into $H_+$, 
and we decorate the curve $L^1_{i}$ with a ``dot''. 
\item
The curves $\gamma_1,\ldots,\gamma_g$ are naturally embedded in $S^3$ together with $\Sigma$. 
Then, for $j\in\{1,\ldots,g\}$, 
let $L^2_{j}$ be the framed knot that is 
the curve $\gamma_j$ equipped with the surface framing of $\Sigma$. 
\end{mystep}
We thus obtain a $(g+k_1)$-component link $L=L_1\sqcup L_2$, 
where $L_1=L^1_{1}\sqcup\cdots\sqcup L^1_{k_1}$ and $L_2=L^2_{1}\sqcup\cdots\sqcup L^2_{g}$. 
This link $L$ is a Kirby diagram representing $X$ (see \cite{GK}), 
and we call this diagram a {\it Kirby diagram obtained from }$(\Sigma;\alpha,\beta,\gamma)$. 
See Figure~\ref{fig:Kirby_diag} for an example, which shows a Kirby diagram for the spin $\mathcal{S}(L(2,1))$ of $L(2,1)$ (see Subsection~\ref{subsec:spin} for the definition of the spin). 
\begin{figure}[t]
\centering
\includegraphics[width=110mm]{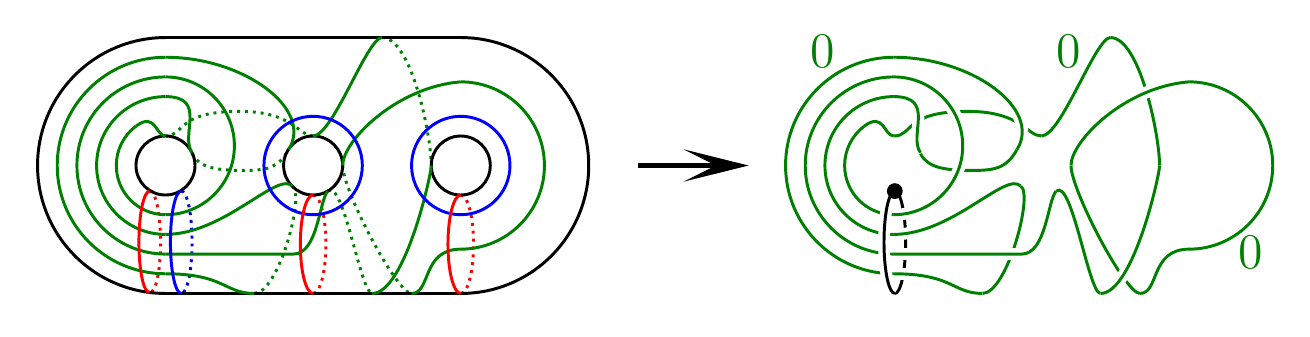}
\caption{An example of Kirby diagram obtained from a trisection diagram of 
$\mathcal{S}(L(2,1))$. }
\label{fig:Kirby_diag}
\end{figure}
\begin{remark}
\label{rmk:nonunique_KD}
A Kirby diagram obtained from $(\Sigma;\alpha,\beta,\gamma)$ 
is not uniquely determined only by the trisection diagram $(\Sigma;\alpha,\beta,\gamma)$. 
Actually, if we choose another embedding of $\Sigma$ in Step~1, 
$L_2$ might change to another framed link $L'_2$. 
However, these $L_2$ and $L'_2$ can be easily related by Kirby moves. 
\end{remark}
\begin{remark}
\label{rmk:permutate_KD}
For a permutation $\tau$ of $\{\alpha,\beta,\gamma\}$, 
a Kirby diagram obtained from $(\mathrm{sgn}(\tau)\Sigma;\tau(\alpha),\tau(\beta),\tau(\gamma))$ 
also describes the same $4$-manifold $X$. 
\end{remark}

If a dotted circle $L^1_{i}$ is a meridian of a framed knot $L^2_{j}$ 
in a Kirby diagram $L=L_1\sqcup L_2$, 
the pair of $L^1_{i}$ and $L^2_{j}$ is a canceling pair, 
and we can reduce the components of the diagram. 
Especially, if there exists a $2$-disk bounded by $L^1_{i}$ 
such that it does not intersect components of $L_2$ except for $L^2_{j}$, 
then the handle cancellation of this pair is done just 
by removing $L^1_{i}$ and $L^2_{j}$ from the diagram without changing the other components. 
We call such a pair an {\it independently canceling pair}. 
If $\alpha_i$ does not intersect curves of $\gamma$ except for $\gamma_j$, 
the pair of $L^1_{i}$ and $L^2_{j}$ is an independently canceling pair. 

\subsection{The Kirby-Thompson invariant of a $4$-manifold}
\label{subsec:Kirby-Thompson_inv}
In this subsection, we review the definition of 
the length of a trisection and the Kirby-Thompson invariant of a closed $4$-manifold. 

\begin{definition}
Let $\Sigma$ be a closed oriented surface of genus $g$. 
The {\it cut complex} $\mathcal{C}$ of $\Sigma$ is a $1$-complex defined as follows:
\begin{itemize}
 \item the vertex set of $\mathcal{C}$ is the set of all the isotopy classes of cut systems of $\Sigma$, and 
 \item two vertices $\alpha,\alpha'\in\Gamma$ are connected by an edge if either one of the following holds:
\begin{itemize}
 \item after reordering, $\alpha_i\P\alpha'_i$ for $i\in\{1,\ldots,g-1\}$, and $\alpha_g$ and $\alpha'_g$ are not isotopic to each other and does not intersect, or 
 \item after reordering, $\alpha_i\P\alpha'_i$ for $i\in\{1,\ldots,g-1\}$, and $\alpha_g\D \alpha'_g$. 
\end{itemize}
\end{itemize}
The edges in the former case are said to be of {\it type-$0$}, and those in the latter case are of {\it type-$1$}. 
\end{definition}
Suppose that $\alpha$ and $\alpha'$ are connected by one type-$0$ edge and that $\alpha_i\P\alpha'_i$ for $i\in\{1,\ldots,g-1\}$. 
Then we will write $\alpha_g\disjoint\alpha'_g$. Note that, in this case, $\alpha_g$ is not isotopic to any curves of $\alpha'$, and $\alpha'_g$ is not isotopic to any curves of $\alpha$. 

We use the symbol $\Gamma_\alpha$ 
to denote the connected component containing $\alpha$ of $\Gamma$ with all the type-1 edges removed. 
We also define $\Gamma_\beta$ and $\Gamma_\gamma$ similarly. 
Note that our notations are slightly different from that of \cite{KT}. 
\begin{remark}
\label{rmk:permutate_CS}
Any two cut systems connected by a type-0 edge can be related by some handle slides. 
Therefore, for any triple $(\alpha',\beta',\gamma')\in\Gamma_\alpha\times\Gamma_\beta\times\Gamma_\gamma$, 
the diagram $(\Sigma;\alpha',\beta',\gamma')$ also represents the same $4$-manifold as $(\Sigma;\alpha,\beta,\gamma)$. 
\end{remark}
We then define a relation between two cut systems belonging to different 
subgraphs $\Gamma_\alpha,\Gamma_\beta,\Gamma_\gamma$. 
\begin{definition}
We say that two cut systems $\alpha$ and $\beta$ are 
{\it good} if we can reorder them so that 
\begin{itemize}
\item
either $\alpha_i\P\beta_i$ or $\alpha_i\D\beta_i$ holds for each $i\in\{1,\ldots,g\}$, and 
\item
$\alpha_i$ and $\beta_j$ do not intersect when $i\ne j$. 
\end{itemize}
We also say that the pair $(\alpha,\beta)$ is {\it good} as well. 
\end{definition}

Let $X$ be a closed connected oriented $4$-manifold and 
$\T=(\Sigma;\alpha,\beta,\gamma)$ a trisection diagram of $X$. 
We here define a {\it loop} in $\mathcal{C}$ as a closed path $\lambda$ in $\mathcal{C}$ satisfying the following three conditions:
\begin{itemize}
\item 
$\lambda$ intersects $\Gamma_\alpha$, $\Gamma_\beta$, and $\Gamma_\gamma$, 
 \item 
there exist vertices 
\[
\alpha_\beta, \alpha_\gamma \in\Gamma_\alpha\cap\lambda, 
\quad 
\beta_\gamma, \beta_\alpha \in\Gamma_\beta\cap\lambda,
\quad 
\gamma_\alpha, \gamma_\beta \in\Gamma_\gamma\cap\lambda
\]
such that the pairs $(\alpha_\beta, \beta_\alpha), (\beta_\gamma, \gamma_\beta)$ and $(\gamma_\alpha, \alpha_\gamma)$ are all good, and 
\item 
the intersection $\Gamma_\alpha\cap\lambda$ (resp. $\Gamma_\beta\cap\lambda$, $\Gamma_\gamma\cap\lambda$) 
is a connected subpath between $\alpha_\beta$ and $\alpha_\gamma$ 
(resp. $\beta_\alpha$ and $\beta_\gamma$, $\gamma_\alpha$ and $\gamma_\beta$). 
\end{itemize}
Note that the pairs $(\alpha_\beta, \beta_\alpha), (\beta_\gamma, \gamma_\beta)$ and $(\gamma_\alpha, \alpha_\gamma)$ 
are connected by $g-k_1$, $g-k_2$ and $g-k_3$ type-1 edges, respectively. 

Let $l_{X, \T}$ denote the minimal length of any loop in $\mathcal{C}$, 
and a loop with length $l_{X, \T}$ is called a {\it realization loop} of $\T$. 
Let us fix a realization loop $\lambda$, and let $l_\alpha(\lambda)$ (resp. $l_\beta(\lambda)$ and $l_\gamma(\lambda)$) 
denote the number of (type-0) edges in $\Gamma_\alpha\cap\lambda$ (resp. $\Gamma_\beta\cap\lambda$ and $\Gamma_\gamma\cap\lambda$). 
Then we set 
\[
L_{X, \T}=l_\alpha(\lambda)+l_\beta(\lambda)+l_\gamma(\lambda)=l_{X, \T}-3g+k_1+k_2+k_3, 
\]
and let $L_X$ denote the minimum number of $L_{X, \T}$ 
over all trisection diagrams $\T$ of $X$.
We call $L_X$ the {\it Kirby-Thompson invariant}.
It was introduced by Kirby and Thompson in \cite{KT}, and it is obviously an invariant of closed connected oriented $4$-manifolds. 

We close this section by presenting a lemma giving a homological relation between two cut systems connected by one type-0 edge. 
It plays a key role in this paper, although the proof is easy. 
\begin{lemma}
\label{lem:0edge_change}
Let $\alpha=(\alpha_1,\ldots,\alpha_g)$ and $\alpha'=(\alpha'_1,\ldots,\alpha'_g)$ 
be two cut systems for $\Sigma$ connected by one type-$0$ edge in $\mathcal{C}$, 
and orient each curve of them arbitrarily. 
Suppose that $\alpha_i\P\alpha'_i$ for $i\in\{1,\ldots,g-1\}$ 
and $\alpha_g\disjoint\alpha'_g$. 
Then there exist $\varepsilon_1,\ldots,\varepsilon_{g-1}\in\{-1,0,1\}$ and $\varepsilon_g\in\{-1,1\}$ such that $\displaystyle [\alpha'_g]=\sum_{i=1}^{g}\varepsilon_i[\alpha_i]$ in $H_1(\Sigma)$. 

\end{lemma}
\begin{proof}
We first recall that $\Sigma$ is oriented. 
Note that $\alpha'_g$ is a simple closed curve in $\Sigma\setminus\bigcup_{i=1}^g\alpha_i$. 
Since $\Sigma\setminus\bigcup_{i=1}^g\alpha_i$ is planar, 
the circle $\alpha'_g$ separates it into two connected components $\Sigma_-\cup\Sigma_+$. 
Orientations of $\Sigma_-$ and $\Sigma_+$ are naturally induced from $\Sigma$. 
Suppose that the orientation of $\alpha'_g$ agrees with that induced from $\Sigma_+$. 
Then the $1$-cycle $\partial[\Sigma_-]$ can be written as 
$\displaystyle \sum_{i=1}^g\varepsilon_i[\alpha_i]-[\alpha'_g]$ 
for some coefficients $\varepsilon_1,\ldots,\varepsilon_g\in\{-1,0,1\}$, 
which is $0$ in $H_1(\Sigma)$. 
If $\varepsilon_g=0$, then 
\[
0=\sum_{i=1}^g\varepsilon_i[\alpha_i]-[\alpha'_g]=
\sum_{i=1}^{g-1}\varepsilon_i[\alpha_i]-[\alpha'_g]
=\sum_{i=1}^{g-1}\varepsilon_i[\alpha'_i]-[\alpha'_g]
\]
in $H_1(\Sigma_g)$. 
It is contrary to the linear independence of $[\alpha'_1],\ldots,[\alpha'_g]$. 
Hence $\varepsilon_g=1$ or $-1$, which concludes the proof.  
\end{proof}

\section{Case $l_\alpha(\lambda)\leq1$}
\label{sec:l_alpha=0and1}
Let $X$ be a closed connected oriented $4$-manifold, 
$\T=(\Sigma;\alpha,\beta,\gamma)$ a $(g;k_1,k_2,k_3)$-trisection diagram
and $\lambda$ a realization loop of $\T$. 
In this section, we prepare some lemmas for 
the cases $l_\alpha(\lambda)=0$ and $l_\alpha(\lambda)=1$, 
which will be used in the proofs of Theorems~\ref{thm:1} and \ref{thm:2}. 

\subsection{The case $l_\alpha(\lambda)=0$.}
\label{subsec:l_alpha=0}
We first consider the case $l_\alpha(\lambda)=0$. 
Suppose that $\alpha=(\alpha_1,\ldots,\alpha_g)$ is the only vertex in $\Gamma_\alpha\cap\lambda$, 
and also suppose that both $\beta$ and $\gamma$ are good for $\alpha$. 
See Figure~\ref{fig:l_a=0} for a schematic picture. 
\begin{figure}[t]
\centering
\includegraphics[width=50mm]{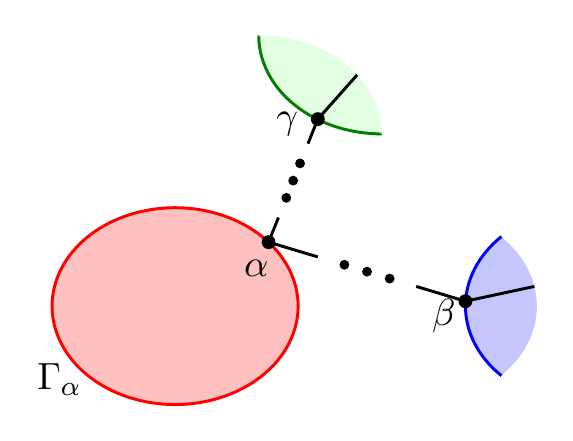}
\caption{How the cut systems are connected in $\mathcal{C}$ when $l_\alpha(\lambda)=0$.}
\label{fig:l_a=0}
\end{figure}
We suppose that the orders of $\beta=(\beta_1,\ldots,\beta_g)$ and 
$\gamma=(\gamma_1,\ldots,\gamma_g)$ satisfy 
\[
\beta_i\P\alpha_i\P\gamma_i,\quad
\beta_i\P\alpha_i\D\gamma_i,\quad
\beta_i\D\alpha_i\P\gamma_i,\quad\text{or}\quad
\beta_i\D\alpha_i\D\gamma_i
\]
for each $i\in\{1,\ldots,g\}$. 

Let $L=L_1\sqcup L_2$ be a Kirby diagram obtained from the $(g;k_1,k_2,k_3)$-trisection diagram $\T$. 
Here $L_1$ consists of $k_1$ dotted circles each $L^1_i$ of which corresponds to $\alpha_i$ such that $\beta_i\P\alpha_i$, 
and $L_2$ consists of $g$ framed knots corresponding to $\gamma_1,\ldots,\gamma_g$. 

We say that a component of a link $K$ in $S^3$ is {\it splittable} in $K$ 
if there exists a $3$-ball of which intersection with $K$ is exactly the component. 
One can easily see that if a dotted circle $L^1_i$ is splittable in $L$, 
$X$ contains $S^1\times S^3$ as a connected summand. 

\begin{proposition}
\label{prop:l_alpha=0}
If $l_\alpha(\lambda)=0$, then $X$ is geometrically simply connected or contains $S^1\times S^3$ as a connected summand. 
\end{proposition}

\begin{proof}
If $\beta_i\P\alpha_i\P\gamma_i$ for some $i\in\{1,\ldots,g\}$, 
$X$ contains $S^1\times S^3$ as a connected summand since $L^1_{i}$ is splittable in $L$. 
Then we suppose that $\beta_i\P\alpha_i\P\gamma_i$ does not hold for any $i\in\{1,\ldots,g\}$ 
in the rest of the proof. 

Suppose $\beta_i\P\alpha_i\D\gamma_i$ for some of $i\in\{1,\ldots,g\}$. 
Then $L^1_i$ and $L^2_i$ form an independently canceling pair since $\alpha$ and $\gamma$ are good. 
Hence we can remove these components $L^1_i$ and $L^2_i$ from $L$ without changing the other components. 
This implies that 
$X$ is described by a Kirby diagram without dotted circles, 
and thus $X$ is geometrically simply connected.

If $\beta_i\D\alpha_i\P\gamma_i$ or $\beta_i\D\alpha_i\D\gamma_i$ for each $i$, 
then it also follows that $X$ is geometrically simply connected. 
\end{proof}
This proposition implies that if $X$ is not geometrically simply connected and and does not contain $S^1\times S^3$ as a connected summand, then $L_X\geq 3$.  
We will strengthen this inequality in Theorem~\ref{thm:1} by investigating the case $l_\alpha=1$. 

\subsection{The case $l_\alpha(\lambda)=1$}
\label{subsec:l_alpha=1}
We next consider the case $l_\alpha(\lambda)=1$. 
The intersection $\Gamma_\alpha\cap\lambda$ is exactly one type-0 edge, 
of which endpoints will be denoted by $\alpha$ and $\alpha'$. 
Suppose that $\alpha$ is good for a vertex $\beta\in\Gamma_\beta$ and 
that $\alpha'$ is good for a vertex $\gamma\in\Gamma_\gamma$. 
See Figure~\ref{fig:l_a=1}.
\begin{figure}[t]
\centering
\includegraphics[width=50mm]{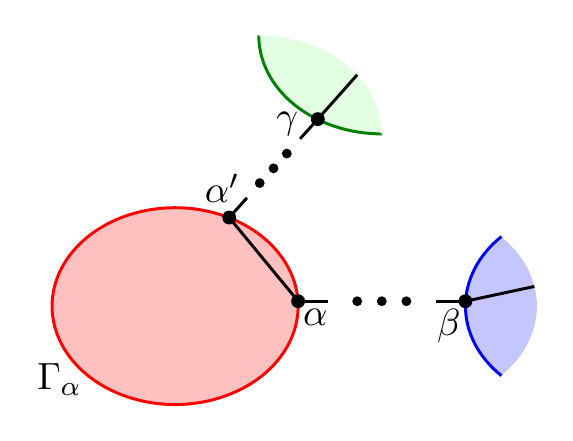}
\caption{How the cut systems are connected in $\mathcal{C}$ when $l_\alpha(\lambda)=1$.}
\label{fig:l_a=1}
\end{figure}
We reorder $\alpha=(\alpha_1,\ldots,\alpha_g)$ and $\alpha'=(\alpha'_1,\ldots,\alpha'_g)$ 
so that $\alpha_i\P\alpha'_i$ for $i\in\{1,\ldots,g-1\}$ and $\alpha_g\disjoint\alpha'_g$. 
We also reorder 
$\beta=(\beta_1,\ldots,\beta_g)$ and $\gamma=(\gamma_1,\ldots,\gamma_g)$ 
so that 
\[
\beta_i\P\alpha_i~\text{ or }~
\beta_i\D\alpha_i,\qquad\text{and}\qquad
\alpha'_i\P\gamma_i~\text{ or }~
\alpha'_i\D\gamma_i
\]
for each $i\in\{1,\ldots,g\}$. 

Let $L=L_1\sqcup L_2$ be a Kirby diagram obtained from the trisection diagram $(\Sigma; \alpha, \beta, \gamma)$. 
Here $L_1$ consists of $k_1$ dotted circles each $L^1_i$ of which corresponds to $\alpha_i$ such that $\beta_i\P\alpha_i$, 
and $L_2$ consists of $g$ framed knots corresponding to $\gamma_1,\ldots,\gamma_g$. 

\begin{lemma}
\label{lem:1}
Suppose that $\beta_i\P \alpha_i \P\alpha_i' \P \gamma_i$ for some $i\in\{1,\ldots,g-1\}$. 
Then $X$ contains $S^1\times S^3$ as a connected summand.
\end{lemma}

\begin{proof}
The dotted circle $L^1_i$ is splittable in $L$, and the lemma follows. 
\end{proof}

By the assumptions on the orders of $\alpha,\alpha',\beta$ and $\gamma$, 
the $g$-th curves of them satisfy one of the following:
\[
\beta_g\P\alpha_g\disjoint\alpha'_g\D\gamma_g,\quad
\beta_g\D\alpha_g\disjoint\alpha'_g\P\gamma_g,\quad
\beta_g\D\alpha_g\disjoint\alpha'_g\D\gamma_g,\quad\text{or}\quad
\beta_g\P\alpha_g\disjoint\alpha'_g\P\gamma_g. 
\]
We discuss these cases in the following lemmas. 

\begin{lemma}
\label{lem:2}
Suppose that $\beta_i\P \alpha_i \P\alpha_i' \P \gamma_i$ does not hold for any $i\in\{1,\ldots,g-1\}$ 
and that $\beta_g\D \alpha_g\disjoint \alpha_g'\P \gamma_g$. 
Then $X$ is geometrically simply connected. 
\end{lemma}
\begin{proof}
For $i\in\{1,\ldots,g-1\}$ with $\beta_i\P\alpha_i$, 
we have $\beta_i\P\alpha_i\P\alpha'_i\D\gamma_i$ by the assumptions. 
For $j\in\{1,\ldots,g\}\setminus\{i\}$, 
$\gamma_j$ does not intersect $\alpha_i$ since $\alpha_i\P\alpha'_i$, 
and hence $L^1_i$ and $L^2_i$ form an independently canceling pair. 
We can cancel all such pairs simultaneously, and it implies that 
$X$ admits a Kirby diagram without dotted circles. 
Thus, $X$ is geometrically simply connected. 
\end{proof}

\begin{lemma}
\label{lem:3}
Suppose that $\beta_i\P \alpha_i \P\alpha_i' \P \gamma_i$ does not hold for any $i\in\{1,\ldots,g-1\}$ 
and that $\beta_g\P \alpha_g\disjoint \alpha_g'\D \gamma_g$.
Then $X$ is geometrically simply connected. 
\end{lemma}

\begin{proof}
Considering a Kirby diagram obtained from the trisection diagram $(-\Sigma;\alpha',\gamma,\beta)$ 
(cf. Remarks~\ref{rmk:permutate_KD} and \ref{rmk:permutate_CS}), 
we can prove this lemma in the same way as Lemma~\ref{lem:2}. 
\end{proof}

\begin{lemma}
\label{lem:4}
Suppose that $\beta_i\P \alpha_i \P\alpha_i' \P \gamma_i$ does not hold for any $i\in\{1,\ldots,g-1\}$ 
and that $\beta_g\D \alpha_g\disjoint \alpha_g'\D \gamma_g$. 
Then $X$ is geometrically simply connected. 
\end{lemma}
\begin{proof}
For each $i\in\{1,\ldots,g-1\}$ such that $\beta_i\P\alpha_i$, 
we have $\beta_i\P\alpha_i\P\alpha'_i\D\gamma_i$ by the assumptions, and 
there exists a $2$-disk $D_i$ bounded by $L^1_i$ such that it intersects $L^2_i$ exactly once. 
Moreover, this disk $D_i$ can be chosen so that it does not intersect $L^2_j$ for $j\in\{1,\ldots,g-1\}\setminus\{i\}$, 
but it possibly intersects $L^2_g$. 
Although the cancelation of $L^1_i\sqcup L^2_i$ might change $L^2_g$ in the Kirby diagram, 
the other components of the diagram do not change by the cancelation. 
Hence all the dotted circles can be removed from the diagram, 
which implies that $X$ is geometrically simply connected. 
\end{proof}

\begin{lemma}
\label{lem:5}
Suppose that $\beta_i\P\alpha_i\P\alpha_i'\D\gamma_i$ or 
$\beta_i\D\alpha_i\P\alpha_i'\P\gamma_i$ for each $i\in\{1,\ldots,g-1\}$
and that $\beta_g\P\alpha_g\disjoint \alpha_g'\P\gamma_g$. 
Then $X$ contains $S^1\times S^3$ as a connected summand. 
\end{lemma}

\begin{proof}
For $i\in\{1,\ldots,g-1\}$ with $\beta_i\P\alpha_i\P\alpha_i'\D\gamma_i$, 
the pair of $L^1_i$ and $L^2_i$ 
is a independently canceling pair, and hence we can remove these components from $L$ without changing the other part. 
We then write $L'$ for the diagram obtained by canceling all such pairs. 
This diagram $L'$ is regarded as a sublink of $L$, and it has a single dotted circle $L^1_g$. 
For $i\in\{1,\ldots,g-1\}$ with $\beta_i\D\alpha_i\P\alpha_i'\P\gamma_i$, 
the component $L^2_i$ is splittable in $L'$. 
The components $L^2_g$ and $L^1_g$ are unlinked since $\alpha_g\disjoint \alpha_g'\P\gamma_g$. 
Therefore, $L^1_g$ is splittable in $L'$, which implies that $X$ contains $S^1\times S^3$ as a connected summand. 
\end{proof}

\begin{remark}
\label{rmk:lemma5}
The assumption in Lemma~\ref{lem:5} actually leads to $X\cong S^1\times S^3$. 
\end{remark}

By Lemmas~\ref{lem:1}, \ref{lem:2}, \ref{lem:3}, \ref{lem:4} and \ref{lem:5}, 
we obtain the following proposition. 

\begin{proposition}
\label{prop:6}
If $X$ is not geometrically simply connected and does not contain $S^1\times S^3$ as a connected summand, 
then the following hold:
\begin{itemize}
\item
$\beta_i\P\alpha_i\P\alpha_i'\P\gamma_i$ does not hold for any $i\in\{1,\ldots,g-1\}$, 
\item
$\beta_i\D\alpha_i\P\alpha_i'\D\gamma_i$ for at least one of $i\in\{1,\ldots,g-1\}$, and 
\item
$\beta_g\P\alpha_g\disjoint \alpha_g'\P\gamma_g$. 
\end{itemize}
\end{proposition}
\section{Proof of Theorem~1}
\label{sec:Thm1}
In this section, we will prove Theorem~\ref{thm:1}. 
Let $X$ be a closed connected oriented $4$-manifold, 
$\T=(\Sigma;\alpha,\beta,\gamma)$ a $(g;k_1,k_2,k_3)$-trisection diagram of $X$
and $\lambda$ a realization loop of $\T$. 
The proof is by contradiction,
so we assume the following: 
 \begin{itemize}
\item[(A1)] $L_X\leq3$. 
\item[(A2)] $X$ is not geometrically simply connected.
\item[(A3)] $X$ does not contain $S^1\times S^3$ as a connected summand. 
 \end{itemize}

If one of $l_\alpha(\lambda)$, $l_\beta(\lambda)$ and $l_\gamma(\lambda)$ is 0, 
it contradicts Proposition~\ref{prop:l_alpha=0}. 
Thus, instead of the condition (A1), we may assume the following:
 \begin{itemize}
\item[(A1)'] $l_\alpha(\lambda)=l_\beta(\lambda)=l_\gamma(\lambda)=1$.
 \end{itemize}

\subsection{Sequences of curves of $\alpha,\beta$ and $\gamma$}
Each of $\Gamma_\alpha\cap\lambda$, $\Gamma_\beta\cap\lambda$ and $\Gamma_\gamma\cap\lambda$ 
contains exactly two vertices, 
which will be denoted by $\alpha,\,\alpha'\in\Gamma_\alpha,\ \beta,\,\beta'\in\Gamma_\beta$ and 
$\gamma,\,\gamma'\in\Gamma_\gamma$. 
Suppose that $(\alpha,\beta)$, $(\alpha',\gamma)$ and $(\beta',\gamma')$ are good. 
Moreover, we can assume that the orders of the curves satisfy the following:
 \begin{itemize}
\item[(A4)] $\alpha_i\P\alpha_i'$ for $i\in\{1, \ldots , g-1\}$, and $\alpha_g\disjoint\alpha_g'$.
\item[(A5)] $\beta_i\D\alpha_i$ for $i\in\{1, \ldots , g-k_1\}$, and otherwise $\beta_i\P\alpha_i$.
\item[(A6)] $\alpha_i'\D\gamma_i$ or $\alpha_i'\P\gamma_i$ for $i\in\{1, \ldots , g\}$.
\item[(A7)] $\beta_i'\P\beta_i$ or $\beta_i'\disjoint \beta_i$ for $i\in\{1, \ldots , g\}$.
\item[(A8)] $\gamma_i'\P\gamma_i$ or $\gamma_i'\disjoint \gamma_i$ for $i\in\{1, \ldots , g\}$.
\item[(A9)] $\gamma_i'\P\beta_{\sigma(i)}'$ or $\gamma_i'\D \beta_{\sigma(i)}'$ for $i\in \{1, \ldots, g\}$ and a fixed permutation $\sigma$ of $\{1, \ldots, g\}$. 
 \end{itemize}
\begin{lemma}
\label{lem:around_disjoint}
If $\gamma_{i}\disjoint\gamma'_{i}$ for some $i\in\{1,\ldots,g\}$, then $\alpha_i\P\gamma_{i}\disjoint\gamma'_{i}\P\beta'_i$. 
\end{lemma}
\begin{proof}
It follows from $l_\gamma(\lambda)=1$ and by the same argument as in Subsection~\ref{subsec:l_alpha=1} 
(cf. Proposition~\ref{prop:6}). 
\end{proof} 

\begin{lemma}
 \label{lemma:7}
$\beta_g\P\alpha_g\disjoint\alpha_g'\P\gamma_g\disjoint\gamma_g'\P\beta_g'\disjoint \beta_g$.
 \end{lemma}
 \begin{proof}
 By Proposition~\ref{prop:6}, we can assume $\beta_g\P\alpha_g\disjoint\alpha_g'\P\gamma_g$.
 Suppose $\beta_g'\P\beta_g$ to obtain a contradiction.
If $\gamma'_{\sigma^{-1}(g)}\P\beta_g'$, this contradicts Lemma~\ref{lem:1}. 
Hence we have $\gamma'_{\sigma^{-1}(g)}\D\beta'_g$. 
If $\gamma_{\sigma^{-1}(g)}\disjoint \gamma'_{\sigma^{-1}(g)}$, it contradicts Lemma~\ref{lem:around_disjoint}. 
Hence $\gamma_{\sigma^{-1}(g)}\P \gamma'_{\sigma^{-1}(g)}$. 
Note that $\gamma_{\sigma^{-1}(g)}$ and $\alpha_g$ intersect transversely at one point. 

Now we consider a Kirby diagram $L=L_1\sqcup L_2$ obtained from $(\Sigma; \alpha, \beta, \gamma)$.
From the assumption~(A5), 
the dotted circles in $L$ 
come only from $\alpha_{g-k_1+1},\ldots, \alpha_g$. 
By Lemma~\ref{lem:1}, we have $\beta_i\P\alpha_i\P\alpha_i'\D\gamma_i$ for $\{g-k_1+1,\ldots, g-1\}$.
Therefore, $L^1_i\sqcup L^2_i$ is an independently canceling pair for $i\in\{g-k_1+1,\ldots, g-1\}$, and $L^1_g$ can be canceled with $L^2_{{\sigma}^{-1}(g)}$.
Hence $X$ is geometrically simply connected.
This is a contradiction.
 \end{proof}

 \begin{lemma}
 \label{lemma:8}
For $i\in\{1, \ldots, g-1\}$, one of the following holds: 
\begin{enumerate}
\item $\sigma(i)=i$ and $\beta_i\P\alpha_i\D\gamma_i\D\beta_i$,
\item $\sigma(i)=i$ and $\beta_i\D\alpha_i\P\gamma_i\D\beta_i$,
\item $\sigma(i)=i$ and $\beta_i\D\alpha_i\D\gamma_i\P\beta_i$,
\item $\sigma(i)=i$ and $\beta_i\D\alpha_i\D\gamma_i\D\beta_i$, or 
\item $\sigma(i)\ne i$ and $\beta_i\D\alpha_i\D\gamma_i\D\beta_{\sigma(i)}\D\alpha_{\sigma(i)}\D\gamma_{\sigma(i)}\D\cdots$. 
\end{enumerate}
 \end{lemma}
 \begin{proof}
Suppose $\sigma(i)=i$. 
The sequence $\beta_i\P\alpha_i\P\gamma_i\P\beta_{\sigma(i)}$ is contrary to Lemma~\ref{lem:1}. 
The sequences $\beta_i\P\alpha_i\P\gamma_i\D\beta_{\sigma(i)}$, $\beta_i\D\alpha_i\P\gamma_i\P\beta_{\sigma(i)}$ and $\beta_i\P\alpha_i\D\gamma_i\P\beta_{\sigma(i)}$ are impossible since $\beta_i=\beta_{\sigma(i)}$. 
Therefore, by Lemma~\ref{lemma:7}, one of (1)-(4) holds if $\sigma(i)=i$. 

Suppose $\sigma(i)\ne i$. 
By Lemma~\ref{lemma:7}, the following are the possible cases:
\begin{align*}
\text{(i)}~&\beta_i\P\alpha_i\P\gamma_i\P\beta_{\sigma(i)},
&\text{(ii)}~&\beta_i\D\alpha_i\P\gamma_i\P\beta_{\sigma(i)},\\
\text{(iii)}~&\beta_i\P\alpha_i\D\gamma_i\P\beta_{\sigma(i)},
&\text{(iv)}~&\beta_i\P\alpha_i\P\gamma_i\D\beta_{\sigma(i)},\\
\text{(v)}~&\beta_i\P\alpha_i\D\gamma_i\D\beta_{\sigma(i)},
&\text{(vi)}~&\beta_i\D\alpha_i\P\gamma_i\D\beta_{\sigma(i)},\\
\text{(vii)}~&\beta_i\D\alpha_i\D\gamma_i\P\beta_{\sigma(i)},
&\text{(viii)}~&\beta_i\D\alpha_i\D\gamma_i\D\beta_{\sigma(i)}.
\end{align*}
The case (i) is contrary to Lemma~\ref{lem:1}. 
If (ii), (iii) or (iv) holds, then $\beta_i$ and $\beta_{\sigma(i)}$ intersect. 
It is a contradiction. 
If (v) or (vi) holds, 
then $\gamma'_i$ intersects both $\beta'_i$ and $\beta'_{\sigma(i)}$, 
which is contrary to the goodness of $(\beta',\gamma')$. 
If (vii) holds, then $\alpha_i$ intersects both $\beta_i$ and $\beta_{\sigma(i)}$. 
It is also a contradiction. 
Therefore, the only case (viii) remains, and the proof is completed. 
 \end{proof}

We suppose that there exists $i\in\{1,\ldots, g-1\}$ such that $\beta_i\P\alpha_i\D\gamma_i\D\beta_i$ (cf. Lemma~\ref{lemma:8}-(1)). 
For any $j\in\{1,\ldots,g-1\}\setminus\{i\}$, 
$\gamma_i\cap\alpha_j$ and $\gamma_i\cap\beta_j$ are empty, 
but $\gamma_i\cap\alpha_g$ and $\gamma_i\cap\beta_g$ might be non-empty. 
Both $\alpha_i$ and $\beta_i$ do not intersect $\gamma$ other than $\gamma_i$. 
Hence $\alpha_i\cup\beta_i\cup\gamma_i$ is contained in a once punctured torus $T_i$, 
in which the other curves (namely $\alpha_g$ and $\beta_g$) lie as shown in the left in Figure~\ref{reddiagram}, for example. 
The curves $\alpha_g$ and $\beta_g$ can be modified so that they are disjoint from $\gamma_i$ 
by handle slides over $\alpha_i$ and $\beta_i$, respectively, 
see the right in the figure. 
 \begin{figure}[t]
 \centering
 \includegraphics[width=0.5\hsize]{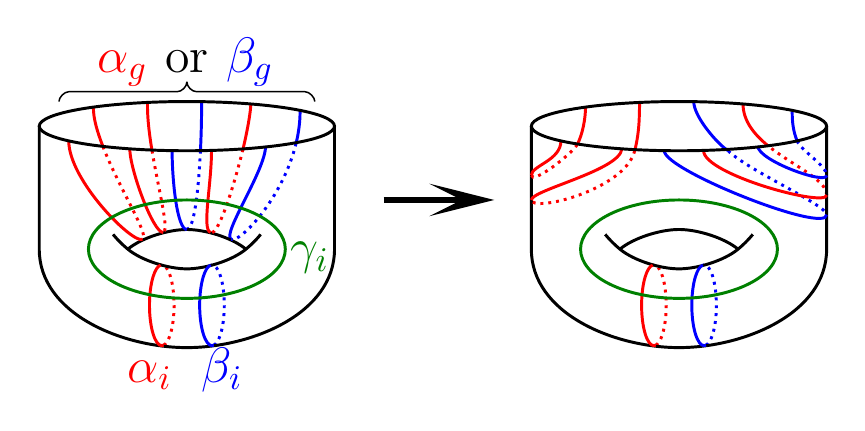}
 \caption{The once punctured torus $T_i$ and handle slides of $\alpha_g$ and $\beta_g$ over $\alpha_i$ and $\beta_i$.}
 \label{reddiagram}
 \end{figure}
By these handle slides, the curves in $\Sigma\setminus T_i$ do not change. 
Then $\alpha_g$ and $\beta_g$ can be isotoped into $\Sigma\setminus T_i$, 
and we can destabilize $\T$ by replacing $T_i$ with a $2$-disk 
(see \cite{GK} for details of the stabilization operation). 
The resulting diagram consists of 
an orientable surface of genus $g-1$ 
and three $(g-1)$-tuples $\tilde{\alpha},\tilde{\beta},\tilde{\gamma}$, 
where $\tilde{\alpha},\tilde{\beta}$ and $\tilde{\gamma}$ are obtained from $\alpha,\beta$ and $\gamma$, respectively, by removing $\alpha_i,\beta_i$ and $\gamma_i$. 
For this diagram, there exists a realization loop $\lambda'$ in the cut complex of 
the genus $g-1$ surface 
such that $l_{\tilde\alpha}(\lambda')=l_{\tilde\beta}(\lambda')=l_{\tilde\gamma}(\lambda')=1$. 

If there exists $i\in\{1,\ldots, g-1\}$ such that $\beta_i\D\alpha_i\P\gamma_i\D\beta_i$ or $\beta_i\D\alpha_i\D\gamma_i\P\beta_i$ (cf. Lemma~\ref{lemma:8}-(2),(3)), 
we can also destabilize the diagram similarly to the previous case. 
Therefore, we can assume in advance that the diagram $(\T;\alpha,\beta,\gamma)$ 
satisfies the following (cf. Lemma~\ref{lemma:8}-(4),(5)):
\begin{itemize}
\item[(A10)] If $\sigma(i)=i$, then $\beta_i\D\alpha_i\D\gamma_i\D\beta_i$. 
\item[(A11)] If $\sigma(i)\ne i$, then $\beta_i\D\alpha_i\D\gamma_i\D\beta_{\sigma(i)}\D\alpha_{\sigma(i)}\D\gamma_{\sigma(i)}\D\cdots$. 
\end{itemize}
\subsection{Algebraic intersection of $\alpha_g$ and curves of $\gamma$}
We fix orientations of the curves of $\alpha$, $\alpha'$ and $\gamma$. 
Note that, for $i\in\{1,\ldots,g-1\}$, 
\begin{align*}
0=\langle \gamma_g, \gamma_i \rangle
=\langle \alpha'_g, \gamma_i\rangle
=\sum_{j=1}^{g} \varepsilon_j \langle \alpha_j, \gamma_i \rangle
= \varepsilon_i + \varepsilon_g\langle \alpha_g, \gamma_i\rangle
\end{align*}
for some $\varepsilon_1,\ldots,\varepsilon_{g-1}\in\{-1,0,1\}$ and $\varepsilon_g\in\{-1,1\}$ by Lemma~\ref{lem:0edge_change}. 
Hence $\langle \alpha_g, \gamma_i\rangle$ is $\pm 1$ or $0$.

Let $L=L_1\sqcup L_2$ denote again a Kirby diagram obtained from $(\Sigma; \alpha, \beta, \gamma)$. 
Note that $L_1$ consists of a single dotted circle $L^1_g$.  
\begin{lemma}
\label{lem:epsilon_g=1}
If there exists $i\in \{1,\ldots, g-1\}$ such that $|\langle \alpha_g, \gamma_i\rangle|=1$, 
then $L^1_g$ and $L^2_i$ form a canceling pair in $L$.
 \end{lemma}
 \begin{proof}
We first recall the construction of a Kirby diagram in Subsection~\ref{subsec:Kirby_diag}: 
we embed the surface $\Sigma$ into $S^3$, 
which is splitted into two handlebodies $H_{+}\cup H_{-}$ by $\Sigma$. 
Note that $\alpha$ and $\alpha'$ play roles of cut systems for $H_{-}$.

Let $D'_1, \ldots , D'_g\subset H_-$ be a meridian disk system of $H_{-}$ such that $\partial D'_j=\alpha'_j$ for $j\in\{1,\ldots,g\}$. 
Then $V=\overline{H_{-}\setminus N\left(\bigcup_{j\neq i} D'_j; H_{-}\right)}$ is a solid torus, 
and there exists a meridian disk $D_g$ of $V$ bounded by $\alpha_g$. 
Since $\gamma_i$ does not intersect $\alpha'_j$ for any $j\in\{1,\ldots, g\}\setminus\{i\}$, 
the component $L^2_i$ lies on $\partial V$. 
By $|\langle \alpha_g, \gamma_i\rangle|=1$, $L^2_i$ can be pushed into $V$ so that it intersects $D_g$ exactly once. 
This implies that $\alpha_g$ and $\gamma_i$ form a canceling pair in $L$. 
 \end{proof}
 
 By Lemma~\ref{lem:epsilon_g=1}, if there exists $i\in \{1,\ldots, g-1\}$ such that $|\langle \alpha_g, \gamma_i\rangle|=1$, $X$ is geometrically simply connected.
 This contradicts the assumption~(A2). 
 Hence, we can assume the following:
\begin{itemize}
 \item[(A12)]
$\langle \alpha_g, \gamma_i\rangle=0$ for any $i\in\{1,\ldots,g-1\}$. 
\end{itemize}
\subsection{Subarcs and wave-moves}
For convenience, let us introduce 
a positive subarc and a negative subarc of $\gamma_i$. 
 Choose $i\in\{1,\ldots,g-1\}$ and suppose that $\alpha_g\cap\gamma_i\neq\emptyset$. 
 We call the closure of a connected component of $\gamma_i \setminus \alpha_g$ a {\it subarc} of $\gamma_i$.
 A subarc $c$ of $\gamma_i$ is said to be {\it positive} if the signs of the intersections of $\gamma_i$ and $\alpha_g$ at $\partial c$ are the same, 
 and otherwise it is {\it negative}. 
 
 \begin{remark}\label{remark:least2}
If $\gamma_i$ has a subarc, 
then $\gamma_i$ has at least two negative subarcs.
 \end{remark}
 
We then introduce the {\it wave-move} that is a modification of a trisection diagram not changing the corresponding $4$-manifold. 
Let $\delta=(\delta_1,\ldots,\delta_g)$ be a cut system of $\Sigma$, 
and let $c$ be an embedded path in $\Sigma$ such that its interior does not intersect $\delta$ 
and its endpoints lie on a single curve of $\delta$, say $\delta_1$. 
Suppose that $c$ intersects either one of two annuli $N(\delta_1;\Sigma)\setminus \delta_1$. 
Such a path $c$ is called a {\it wave} of $\delta_1$. 
The simple closed curve $\delta_1$ is separated into two arcs $d$ and $d'$ by $\partial c$, 
and either $d\cup c$ or $d'\cup c$ is a simple closed curve null-homologous in $\Sigma/\bigcup_{i=2}^g \delta_i$. 
If $d\cup c$ is so, then the arcs $d$ and $d'$ are said to be {\it non-preferred} and {\it preferred} for $c$, respectively. 
It is clear that the simple closed curve $d'\cup c$ is isotopic into $\Sigma\setminus\bigcup_{i=1}^g \delta_i$. 
Then $\delta$ with $\delta_1$ replaced with $d'\cup c$ is also a cut system, 
and this replacement is called the {\it wave-move of $\delta$ along $c$} (see \cite{HOT80}). 
It is clear that the resulting cut system belongs to the same subgraph $\Gamma_\delta$ as $\delta$
in the cut complex $\mathcal{C}$. 
 
If there exists a negative subarc $c$ of $\gamma_i$ that does not intersect both $\alpha_i$ and $\beta_{\sigma(i)}$, 
then $c$ can be seen as a wave of $\alpha_g$ and $\beta_g$. 
Then we can perform the wave-moves of $\alpha$ and $\beta$ along $c$, 
so that the number of subarcs of $\gamma_i$ decreases. 
Since the total of subarcs of the curves of $\gamma$ is finite, and by Remark~\ref{remark:least2}, 
we can assume the following:
\begin{itemize}
 \item[(A13)]
For $i\in\{1,\ldots,g-1\}$, 
if $\gamma_i$ has a subarc, 
then $\gamma_i$ has exactly two negative subarcs each of which intersects either $\alpha_i$ or $\beta_{\sigma(i)}$. 
\end{itemize}

The following lemma says that $\gamma_i$ has no subarcs if $\sigma(i)=i$. 
 \begin{lemma}\label{lemma:13}
For $i\in\{1,\ldots,g-1\}$, if $\sigma(i)=i$, then $\gamma_i\cap \alpha_g=\emptyset$.
 \end{lemma}
 \begin{proof}
Fix $i\in\{1,\ldots,g-1\}$. 
Suppose $\sigma(i)=i$ and $\gamma_i\cap \alpha_g\neq \emptyset$ to obtain a contradiction. 
Recall the assumption~(A10), which says that 
$\gamma_i$ intersects each of $\alpha_i$ and $\beta_{i}$ exactly once. 

By the assumption (A13), there exist negative subarcs $c$ and $c'$ of $\gamma_i$ that intersect $\alpha_i$ and $\beta_i$, respectively. 
The other subarcs of $\gamma_i$ are positive. 
Then $c$ and $c'$ must lie in the opposite sides of $\alpha_g$. 
More precisely, $c$ and $c'$ can not intersect the same annulus of $N(\alpha_g; \Sigma)\setminus \alpha_g$. 
We suppose that $\alpha_g$ and $\beta_g$ are the same curve in $\Sigma$, 
then $c$ can be regarded as a wave of $\beta_g$. 
Let $b$ be one of the component $\beta_g\setminus\partial c$, 
and suppose that it is non-preferred. 
It is easy to see that $b\cup c$ is a simple closed curve homologous to $\beta_i$. 
Then there exists a subsurface $\tilde\Sigma$ in $\Sigma$ with $\partial\tilde\Sigma=\beta_i\sqcup(b\cup c)$ after slightly perturbing the curves. 
Since $c'$ intersects $\beta_i$, one of the endpoints of $c'$ is contained in $b$. 
Hence, $c$ and $c'$ lie in the same side of $\alpha_g$, 
which is a contradiction. 
 \end{proof}
We finally give the proof of Theorem~1. 
In the proof below, 
we will find wave-moves modifying $\alpha_g$ and $\beta_g$ so that they do not intersect any curve of $\gamma$. 
If we can find such moves, 
it is contrary to the assumption (A3), and the proof will be completed. 
By Lemma~\ref{lemma:13}, what we remain to consider is the case $\sigma(i)\neq i$. 
 \begin{proof}[Proof of Theorem~1]
Fix $i\in\{1,\ldots,g-1\}$ with $\sigma(i)\neq i$. 
Suppose that $\gamma_i$ has a subarc $c$ intersecting $\alpha_i$, and then $c$ is negative. 
We also suppose that $\alpha_g$ and $\beta_g$ are the same curve in $\Sigma$, and we 
regard $c$ as a wave of $\beta_g$. 
Let $b$ be one of the component $\beta_g\setminus\partial c$, and suppose that it is non-preferred. 
Then there exists a subsurface $\tilde\Sigma$ in $\Sigma$ with $\partial\tilde\Sigma=\beta_i\sqcup(b\cup c)$. 
Note that there exists a subarc $c'$ of $\gamma_{\sigma^{-1}(i)}$ intersecting $\beta_i$, 
and hence one of the endpoints of $c'$ is on $b$. 
Set $c''=\partial(\tilde\Sigma\cup N(\alpha_i\cup\beta_i;\Sigma))\setminus\mathrm{Int}\,b$. 
Then $c''$ is an arc intersecting curves of $\gamma$ exactly once with $\gamma_{\sigma^{-1}(i)}$, 
and it is a wave of $\alpha_g$ and $\beta_g$. 
Hence the wave-moves of $\alpha_g$ and $\beta_g$ along $c''$ do not create new subarcs, 
and the number of subarcs decreases. 
Since the total of subarcs is finite, 
we can modify $\alpha_g$ and $\beta_g$ so that they do not intersect any curve of $\gamma$. 
This implies that $X$ contains $S^1\times S^3$ as a connected summand, which is contrary to (A3). 
 \end{proof}
\section{Theorem~\ref{thm:2} and the spins of $3$-manifolds}
\label{sec:Thm2}
In this section, we prove Theorem~\ref{thm:2} and 
determine the exact value of $L_{\mathcal{S}(L(2, 1))}$ 
as an application of Theorem~\ref{thm:2}. 
\subsection{Proof of Theorem~\ref{thm:2}}
Let $X$ be a closed connected oriented $4$-manifold.
Let $L=L_1\sqcup L_2$ be a Kirby diagram obtained from a $(g;k_1,k_2,k_3)$-trisection diagram $\mathcal{T} = (\Sigma; \alpha, \beta, \gamma)$ of $X$, 
where $L_1$ consists of $k_1$ dotted circles each $L^1_i$ of which corresponds to $\alpha_i$ such that $\beta_i\P\alpha_i$, 
and $L_2$ consists of $g$ framed knots corresponding to $\gamma_1,\ldots,\gamma_g$ (see Section 2).
It is enough to prove the following statement that is equivalent to Theorem~\ref{thm:2}.

\begin{theorem}[Theorem~\ref{thm:2}]
\label{thm:2:Section5}
Suppose that $X$ does not contain $S^1 \times S^3$ as a connected summand.
If $L_X \leq 5$, then $\pi_1(X)$ is isomorphic to the trivial group or the infinite cyclic group.
\end{theorem}
\begin{proof}
Let $\lambda$ be a realization loop of $\mathcal{T}$.
If $L_X \leq 5$, then at least one of $l_{\alpha}(\lambda), l_{\beta}(\lambda)$ or $l_\gamma(\lambda)$ is less than or equal to $1$.
We suppose that $l_{\alpha}(\lambda) \leq 1$ without loss of generality.
The case of $l_{\alpha}(\lambda) = 0$ is showed by Proposition \ref{prop:l_alpha=0}.
We only consider the case of $l_{\alpha}(\lambda) = 1$.
Let $\alpha = (\alpha_1, \dots, \alpha_g)$ and $\alpha' = (\alpha'_1, \dots, \alpha'_g)$ be the vertices in $\Gamma_\alpha \cap \lambda$, which are connected by exactly one type-$0$ edge. 
Let $\beta = (\beta_1, \dots, \beta_g)$ and $\gamma = (\gamma_1, \dots, \gamma_g)$ be vertices in $\Gamma_\beta \cap \lambda$ and $\Gamma_\gamma \cap \lambda$, respectively, such that $(\beta, \alpha)$ and $(\gamma, \alpha')$ are good. 
We suppose that $\alpha_i \P \alpha'_i$ for $i \in \{ 1, \dots, g-1 \}$ and $\alpha_g \disjoint \alpha'_g$.
For $i \in \{ 1, \dots, g-1 \}$, either one of the following holds:
\[
\beta_i \P \alpha_i \P \alpha'_i \P \gamma_i,\quad
\beta_i \P \alpha_i \P \alpha'_i \D \gamma_i,\quad
\beta_i \D \alpha_i \P \alpha'_i \P \gamma_i,\quad\text{or}\quad
\beta_i \D \alpha_i \P \alpha'_i \D \gamma_i.
\]
If there exists $i \in \{1, \dots ,g-1\}$ such that $\beta_i \P \alpha_i \P \alpha'_i \P \gamma_i$, 
then $X$ has $S^1 \times S^3$ as a connected summand by Lemma~\ref{lem:1}. 
It is a contradiction, so we will suppose in the rest of the proof that 
$\beta_i \P \alpha_i \P \alpha'_i \P \gamma_i$ does not hold for any $i \in \{1, \dots ,g\}$. 

The following four cases about $g$-th curves are possible: 
\[
\beta_g\P\alpha_g\disjoint\alpha'_g\D\gamma_g,\quad
\beta_g\D\alpha_g\disjoint\alpha'_g\P\gamma_g,\quad
\beta_g\D\alpha_g\disjoint\alpha'_g\D\gamma_g,\quad\text{or}\quad
\beta_g\P\alpha_g\disjoint\alpha'_g\P\gamma_g. 
\]
In the first three cases, we have $\pi_1(X)\cong\{1\}$ by Lemmas~\ref{lem:2}, \ref{lem:3} and \ref{lem:4}. 

Suppose $\beta_g\P\alpha_g\disjoint\alpha'_g\P\gamma_g$ and 
set $\mathcal{I}=\{i\mid\beta_i \P \alpha_i \P \alpha'_i \D \gamma_i\}$. 
Note that $L$ contains a dotted circle $L^1_i$ if $i\in\mathcal{I}$. 
By the goodness of $(\alpha',\gamma)$, the pair of $L^1_i$ and $L^2_i$ is an independently canceling pair for $i\in\mathcal{I}$. 
After removing all such pairs from $L$, the remaining dotted circle is only $L^1_g$. 

We give orientations to the curves in $\alpha, \alpha'$ and $\gamma$ so that $\langle \alpha'_i, \gamma_i \rangle \geq 0$ for $i \in \{1, \dots, g \}$.
By Lemma \ref{lem:0edge_change}, for any $i \in \{1, \dots, g\}$,
\begin{align*}
\langle \alpha_g, \gamma_i \rangle 
=  \sum_{j=1}^{g}\varepsilon_j\langle \alpha'_j , \gamma_i \rangle = \varepsilon_i,
\end{align*}
for some $\varepsilon_1, \dots, \varepsilon_{g-1} \in \{-1, 0, 1\}$ and $\varepsilon_{g} \in \{-1, 1\}$.
Note that $\langle \alpha_g, \gamma_i \rangle=\varepsilon_i$ coincides with the linking number of $L^{1}_g$ and $L^{2}_i$.
Therefore, we have 
\[
\pi_1(X) \cong \left\langle \alpha_g ~\middle|~ \alpha_g^{\varepsilon_i} ,\ i\in\{1,\ldots,g\}\setminus\mathcal{I} \right\rangle,
\]
which is trivial or infinite cyclic.
\end{proof}

\subsection{The spin of a $3$-manifold and its Kirby-Thompson invariant}
\label{subsec:spin}
We review the spin of a $3$-manifold. 
Let $M$ be a closed connected oriented $3$-manifold and $\mathring{M}$ the closure of $M \setminus D^3$. 
We call $\partial(\mathring{M} \times D^2)$ the {\it spin} of $M$, which is denoted by $\mathcal{S}(M)$.
In~\cite{M}, Meier constructed a trisection and a trisection diagram of $\mathcal{S}(M)$ by using a Heegaard splitting of $M$.
Here we focus on the spin $\mathcal{S}(L(p, q))$ of the lens space $L(p,q)$.
Kirby and Thompson gave an upper bound $6$ for $L_{\mathcal{S}(L(2, 1))}$ in~\cite{KT}, and we can easily obtain the following by using their technique. 
\begin{proposition}[cf. \cite{KT}]\label{prop:KirbyThompson}
For any coprime integers $p$ and $q$ with $p \geq 2$, $L_{\mathcal{S}(L(p, q))} \leq 3p$.
\end{proposition}

By applying Theorem~\ref{thm:2:Section5}, we have $L_{\mathcal{S}(L(p, q))} \geq 6$ since 
$\pi_1(\mathcal{S}(L(p, q))) \cong \mathbb{Z} / p\mathbb{Z}$.
In particular, we have shown the following. 
\begin{corollary}
\label{cor:S(L(2,1))}
$L_{\mathcal{S}(L(2,1))} = 6$. 
\end{corollary}
\section{Proof of Theorem~3}
\label{sec:Thm3}
This section is devoted to proving Theorem~\ref{thm:3}. 
\subsection{Setting of cut systems}
We start with giving some notations and assumptions on cut systems for the proof of Theorem~\ref{thm:3}. 
Let $X$ be a closed connected oriented $4$-manifold with $2\leq|H_1(X)|<\infty$. 
Let $\T=(\Sigma;\alpha,\beta,\gamma)$ be a $(g;k_1,k_2,k_3)$-trisection diagram and 
$\lambda$ a realization loop of $\mathcal{T}$. 
Set $p=|H_1(X)|$. 
Here we can assume $g\geq3$ by the classification in \cite{MZ17} (cf. \cite{MSZ16}). 
Suppose that $l_\alpha(\lambda)=\min\{l_\alpha(\lambda),l_\beta(\lambda),l_\gamma(\lambda)\}$ 
and that $\alpha$ is good for $\beta$. 
Here we can assume that $l_\alpha(\lambda)\geq1$ by Proposition~\ref{prop:l_alpha=0}. 
We write $m=l_\alpha(\lambda)$ and $k=k_1$ for the convenience. 
The intersection $\Gamma_\alpha\cap\lambda$ contains $m+1$ vertices, 
which will be denoted by $\alpha^{(0)}, \alpha^{(1)},\ldots, \alpha^{(m)}$. 
Here $\alpha^{(m)}=\alpha$, and 
$\alpha^{(h-1)}$ and $\alpha^{(h)}$ are connected by exactly one \mbox{type-$0$} edge for each $h\in\{1,\ldots,m\}$. 
We also suppose that $\alpha^{(0)}$ and $\gamma$ are good. 
See Figure~\ref{fig:Thm3}. 
\begin{figure}[t]
\centering
\includegraphics[width=73mm]{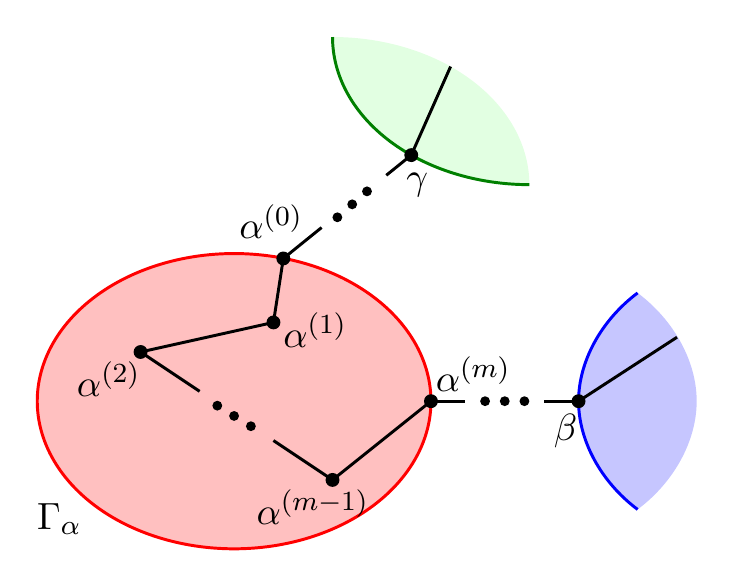}
\caption{How the cut systems are connected in $\mathcal{C}$.}
\label{fig:Thm3}
\end{figure}

We reorder the curves of $\alpha^{(m)}$ and $\beta$ so that 
$\alpha^{(m)}_i\P\beta_i$ for $i\in\{1,\ldots,k\}$ and 
$\alpha^{(m)}_i\D\beta_i$ for $i\in\{k+1,\ldots,g\}$. 
Since $\alpha^{(m)}$ and $\alpha^{(m-1)}$ are connected by one \mbox{type-0} edge, 
there exists a unique curve in $\alpha^{(m)}$ that is not parallel to any curve in $\alpha^{(m-1)}$. 
The suffix of this curve will be denoted by $i^{(m)}$, 
and the suffices of $\alpha^{(m-1)}$ can be set 
so that $\alpha_{i}^{(m-1)}\P\alpha_{i}^{(m)}$ for $i\in\{1,\ldots,g\}\setminus\{i^{(m)}\}$ 
and $\alpha_{i^{(m)}}^{(m-1)}\disjoint\alpha_{i^{(m)}}^{(m)}$. 
Then we reorder $\alpha^{(m-2)},\ldots,\alpha^{(0)}$ in the same manner,
so that for each $h\in\{1,\ldots,m\}$, there exists a unique $i^{(h)}\in\{1,\ldots,g\}$ such that 
$\alpha_{i}^{(h-1)}\P\alpha_{i}^{(h)}$ for $i\in\{1,\ldots,g\}\setminus\{i^{(h)}\}$
and $\alpha_{i^{(h)}}^{(h-1)}\disjoint\alpha_{i^{(h)}}^{(h)}$. 

We then give orientations of the curves of $\alpha^{(0)}$ and $\gamma$ so that $\langle\alpha_i^{(0)}, \gamma_j\rangle$ is $0$ or $1$.
Note that such orientations exist by the goodness of $\alpha^{(0)}$ and $\gamma$. 
Moreover, orientations of $\alpha^{(1)},\ldots,\alpha^{(m)}$ are given inductively in the following claim.
\begin{claim}
\label{clm:entry_change}
There exist orientations of the curves of $\alpha^{(1)},\ldots,\alpha^{(m)}$ such that 
\[
\left[\alpha^{(h)}_i\right] = 
\begin{cases}
\left[\alpha^{(h-1)}_i\right] & (i\ne i^{(h)}) \\[2mm]
\displaystyle \left[\alpha^{(h-1)}_{i^{(h)}}\right]+\sum_{l\ne i^{(h)}}\varepsilon^{(h)}_{l}\left[\alpha^{(h-1)}_l\right] & (i = i^{(h)}), 
\end{cases}
\]
in $H_1(\Sigma)$, for any $h\in\{1,\ldots,m\}$ and for some $\varepsilon^{(h)}_{l}\in\{-1,0,1\}$.  
\end{claim}
\begin{proof}
It follows from Lemma~\ref{lem:0edge_change}. 
\end{proof} 
 
\subsection{Intersection matrices} 
We investigate the differences of cut systems $\alpha^{(0)},\ldots,\alpha^{(m)}$ 
by using matrices having algebraic intersection numbers with $\gamma$ as their entries: 
for $h\in\{0,1,\ldots,m\}$, let us define a $g\times g$-matrix 
\[
I^{(h)}
=\left(
\begin{array}{ccc}
\langle\alpha_1^{(h)}, \gamma_1\rangle & \cdots & \langle\alpha_1^{(h)}, \gamma_g\rangle\\
\vdots & \ddots & \vdots\\
\langle\alpha_g^{(h)}, \gamma_1\rangle & \cdots & \langle\alpha_g^{(h)}, \gamma_g\rangle
\end{array}
\right). 
\]
Let $\mathbf{u}^{(h)}_i=\left( \langle\alpha_i^{(h)}, \gamma_1\rangle, \ldots, \langle\alpha_i^{(h)}, \gamma_g\rangle  \right)$ 
denote the $i$-th row vector of $I^{(h)}$. 
\begin{claim}
\label{clm:row_change}
For $h\in\{1,\ldots,m\}$, 
\[
\mathbf{u}^{(h)}_i = 
\begin{cases}
\mathbf{u}^{(h-1)}_i & (i\ne i^{(h)}) \\[2mm]
\displaystyle \mathbf{u}^{(h-1)}_{i^{(h)}}+\sum_{l\ne i^{(h)}}\varepsilon^{(h)}_{l}\mathbf{u}^{(h-1)}_l & (i = i^{(h)}). 
\end{cases}
\]
\end{claim}
\begin{proof}
It follows from Claim~\ref{clm:entry_change}. 
\end{proof}
This claim tells us the difference between $I^{(h-1)}$ and $I^{(h)}$, which lies only in their \mbox{$i^{(h)}$-th} rows. 
Note that $I^{(0)}$ has at most one non-zero entry, which is exactly $1$, 
in each row and in each column since $\alpha^{(0)}$ and $\gamma$ are good. 
The matrix $I^{(0)}$ is in some sense the ``simplest'' among $I^{(0)},\ldots,I^{(m)}$, 
and we can consider the other matrices to be obtained from $I^{(0)}$ in order by changing rows according to Claim~\ref{clm:row_change}.

For $h\in\{0,\ldots,m\}$, define a map $\rho_h:\{1,\ldots,g\}\to\{0,\ldots,h\}$ 
by the following: 
\begin{itemize}
 \item
$\rho_0(i)=0$ for any $i\in\{1,\ldots,g\}$, and 
 \item 
$\rho_h(i)=
\begin{cases}
h             & (i  = i^{(h)})\\
\rho_{h-1}(i) & (i\ne i^{(h)}).
\end{cases}$
\end{itemize}
Note that the conditions $i=i^{(h)}$ and $i\ne i^{(h)}$ above are equivalent to 
$\alpha^{(h-1)}_i\disjoint \alpha^{(h)}_i$ and to $\alpha^{(h-1)}_i\P \alpha^{(h)}_i$, respectively. 
We also note that $\rho_h$ is not necessarily surjective, 
but it is injective on $\{i\mid\rho_h(i)\geq1\}$. 

We immediately notice that $\mathbf{u}^{(h)}_i=\mathbf{u}^{(h-1)}_i$ if $\rho_{h}(i)=\rho_{h-1}(i)$. 
Roughly speaking, it means that the $i$-th row $\mathbf{u}^{(*)}_i$ does not change unless $\rho_{*}(i)$ changes. 
More generally, the following holds.
\begin{claim}
\label{clm:rho_eq}
Fix $h\in\{0,\ldots,m\}$ and $i\in\{1,\ldots,g\}$. 
Then $\mathbf{u}^{(h')}_i =\mathbf{u}^{(\rho_h(i))}_i$ for any 
$h'\in\{\rho_h(i),\ldots,h\}$. 
\end{claim}
\begin{proof}
It follows from Claim~\ref{clm:row_change}. 
\end{proof}

Let $\{F_n\}_{n=1}^{\infty}$ denote the Fibonacci $g$-step sequence, 
that is, it is defined by $\displaystyle F_n=\sum_{i=1}^{g}F_{n-i}$ with initial values $F_0=F_1=1$ and $F_n=0$ for $n<0$. 
The following claim will be used to estimate the norms of vectors in matrices. 
\begin{claim}
\label{clm:maximal_entry}
For $h\in\{1,\ldots,m\}$ and $i\in\{1,\ldots,g\}$, the following hold. 
\begin{enumerate}
 \item 
If $\rho_h(i)=0$, 
the entries in $\mathbf{u}^{(h)}_i$ are $0$ except at most one entry, 
which is $1$. 
 \item 
If $\rho_h(i)\geq1$, 
the absolute values of the entries in $\mathbf{u}^{(h)}_i$ are 
at most $F_{\rho_h(i)}$. 
 \item 
If $\rho_h(i)\geq1$, 
the absolute values of the entries in $\mathbf{u}^{(h)}_i$ are 
at most $2^{\rho_h(i)-1}$. 
\end{enumerate}
\end{claim}
\begin{proof}
\noindent(1) 
It follows from Claim~\ref{clm:rho_eq} 
and the goodness of $\alpha^{(0)}$ and $\gamma$. 

\noindent(2) 
Fix $j\in\{1,\ldots,g\}$ arbitrarily. 
By Claim~\ref{clm:rho_eq}, 
it is enough to check the claim for $h\in\{1,\ldots,m\}$ and $i=i^{(h)}$, 
so we prove $\left|\left\langle\alpha_{i^{(h)}}^{(h)},\gamma_j\right\rangle\right|\leq F_{h}$. 
The proof will proceed by induction on $h$. 

First consider the case $h=1$. 
Then $\left|\left\langle\alpha_{i^{(1)}}^{(1)},\gamma_j\right\rangle\right|$ is $0$ or $1$ by (1) and Claim~\ref{clm:row_change}, 
which is the desired conclusion. 

We next consider the case $h\in\{2,\ldots,m\}$ and 
suppose that 
$\left|\left\langle\alpha_{i}^{(h')},\gamma_j\right\rangle\right|\leq F_{\rho_{h'}(i)}$ 
holds for $h'\in\{1,\ldots,h-1\}$. 
Define two sets $\Lambda_0,\Lambda_1\subset\{1,\ldots,g\}$ as 
\[
\Lambda_0=\{i\mid\rho_{h-1}(i)=0 \}
\quad\text{and}\quad
\Lambda_1=\{i\mid\rho_{h-1}(i)\geq1 \}. 
\]
Then we have 
$|\Lambda_1|\leq \min\{g,h-1\}$ and 
$\left|\left\langle\alpha_{i}^{(h-1)},\gamma_j\right\rangle\right|\leq F_{\rho_{h-1}(i)}$ for $i\in\Lambda_1$. 
Note that we have made no assumption whether $h$ or $g$ is larger. 
Moreover, by Claim~\ref{clm:rho_eq} and the goodness of $\alpha^{(0)}$ and $\gamma$, 
there exists at most one $i\in \Lambda_0$ such that 
$\left|\left\langle\alpha_{i}^{(h-1)},\gamma_j\right\rangle\right|=1$ and 
$\left|\left\langle\alpha_{i'}^{(h-1)},\gamma_j\right\rangle\right|=0$ for $i'\in\Lambda_0\setminus\{i\}$. 
Therefore, by Claim~\ref{clm:row_change}, we have 
\begin{align*}
\left|\left\langle\alpha_{i^{(h)}}^{(h)},\gamma_j\right\rangle\right|
&=\left|\left\langle\alpha_{i^{(h)}}^{(h-1)},\gamma_j\right\rangle + \sum_{k\ne i_h}\varepsilon^{(h-1)}_{k}\left\langle\alpha_{k}^{(h-1)},\gamma_j\right\rangle  \right|
\leq \sum_{k=1}^{g}\left| \left\langle\alpha_{k}^{(h-1)},\gamma_j\right\rangle \right|\\
&\leq1+\sum_{k\in \Lambda_1} F_{\rho_{h-1}(k)}
\leq 1+\sum_{l=1}^{\min\{g,h-1\}} F_{h-l}=F_h, 
\end{align*}
which concludes the proof. 

\noindent(3)
It is obvious from (2) and $F_k\leq 2^{k-1}$. 
\end{proof}

\subsection{Submatrices and Proof of Theorem~\ref{thm:3}}
The matrix $I^{(m)}$ presents a homological relation between $\alpha^{(m)}$ and $\gamma$. 
Indeed, the homology group $H_1(X)$ can be computed by a submatrix of $U$ as we will see in Claim~\ref{clm:invariant_factor}. 
We give the proof of Theorem~\ref{thm:3} at the end of this subsection. 

Set 
$U=\left(\begin{array}{c}
\mathbf{u}^{(m)}_{1}\\ \vdots\\ \mathbf{u}^{(m)}_{k}
\end{array}\right)$, 
which is a $k\times g$-matrix consisting of the upper $k$ rows of $I^{(m)}$. 

\begin{claim}
\label{clm:invariant_factor}
The product of all the invariant factors of $U$ coincides with $p$. 
\end{claim}
\begin{proof}
A Kirby diagram obtained from $(\Sigma;\alpha^{(m)},\beta,\gamma)$ 
consists of $k$ dotted circles $L^1_1,\ldots,L^1_k$ and $g$ framed knots $L^2_1,\ldots,L^2_g$.
The handle decomposition described by this diagram defines a chain complex, 
and the matrix representation of the boundary operator on $2$-chains is given by $U$ 
since the linking number of $L^1_i$ and $L^2_j$ is equal to $\langle\alpha_i^{(m)}, \gamma_j\rangle$.
Hence the claim follows by $|H_1(X)|=p$. 
\end{proof}

Let $\mathbf{v}_i$ be the $i$-th column vector of $U$, and then 
$U$ is written as $U=\Big(\mathbf{v}_1 \cdots \mathbf{v}_g\Big)$. 
\begin{claim}
\label{clm:deteminant}
There exists $\{i_1,\ldots,i_k\}\subset\{1,\ldots,g\}$ such that 
$p\leq\left|\det \Big(\mathbf{v}_{i_1} \cdots \mathbf{v}_{i_k} \Big)\right|$. 
\end{claim}
\begin{proof}
The claim follows from the fact that the product $p$ of all the invariant factors of $U$ is equal to the greatest common divisor of all $k\times k$ minors of $U$. 
\end{proof}
Actually, for any $\{i_1,\ldots,i_k\}\subset\{1,\ldots,g\}$, 
if $\det \Big(\mathbf{v}_{i_1} \cdots \mathbf{v}_{i_k} \Big)\ne0$, 
then we have $p\leq\left|\det \Big(\mathbf{v}_{i_1} \cdots \mathbf{v}_{i_k} \Big)\right|$. 
Let us fix such a matrix and write $U_0=\Big(\mathbf{v}_{i_1} \cdots \mathbf{v}_{i_k} \Big)$. 
We then prove Theorem~\ref{thm:3}. 

\begin{proof}[Proof of Theorem~\ref{thm:3}]
Let $\mathbf{u}_i$ be the $i$-th row vector of $U_0$, and then 
$U_0$ will be written as 
$U_0=\left(\begin{array}{c}
\mathbf{u}_{1}\\ \vdots\\ \mathbf{u}_{k}
\end{array}\right)$. 
Note that $\mathbf{u}_{1},\ldots,\mathbf{u}_{k}$ are non-zero vectors since $\det U_0\ne0$.
It is enough to prove that $C\sqrt{\,\log p}\leq m$ for some $C>0$ since $3m\leq L_X$. 
The proof is divided into two cases (i) $k>m$ and (ii) $k\leq m$. 

\vspace{2mm}
(i) Suppose $k>m$. Define two sets $\Lambda'_0,\Lambda'_1\subset\{1,\ldots,k\}$ as 
\[
\Lambda'_0=\{i\mid\rho_{m}(i)=0 \}
\quad\text{and}\quad
\Lambda'_1=\{i\mid\rho_{m}(i)\geq1 \}. 
\]
Then $|\Lambda'_0|\geq k-m>0$, 
and $\mathbf{u}_{i}$ is a standard unit vector for $i\in\Lambda'_0$ by Claim~\ref{clm:maximal_entry}-(1). 
We also have $1\leq|\Lambda'_1|\leq m$. 
For $i\in\Lambda'_1$, set 
$\displaystyle \mathbf{u}'_{i}=\mathbf{u}_{i}-\sum_{l\in\Lambda'_0}(\mathbf{u}_{i}\cdot\mathbf{u}_{l})\,\mathbf{u}_{l}$, 
where $\cdot$ is the dot product. 
Then, for $i\in\Lambda'_1$, $|\Lambda'_1|$ entries of $\mathbf{u}'_{i}$ are the same as those of $\mathbf{u}_{i}$, 
and the others are $0$. 
Let $U_0'$ be the matrix obtained from $U_0$ by replacing $\mathbf{u}_{i}$ with $\mathbf{u}'_{i}$ for all $i\in\Lambda'_1$, and then clearly $\det U_0'=\det U_0$. 
Moreover, the absolute values of the non-zero entries of $\mathbf{u}'_{i}$ is at most $2^{\rho_m(i)-1}$ by Claim~\ref{clm:maximal_entry}-(3), 
and we have $\|\mathbf{u}'_{i}\|\leq 2^{\rho_m(i)-1}|\Lambda'_1|^{1/2}\leq 2^{\rho_m(i)-1}m^{1/2}$. 
By the Hadamard inequality, we have
\begin{align*}
\det U'_0 
\leq \prod_{i\in\Lambda'_0}\|\mathbf{u}_{i}\| \prod_{i\in\Lambda'_1}\|\mathbf{u}'_{i}\|
\leq \prod_{i\in\Lambda'_1}2^{\rho_m(i)-1}m^{1/2}
\leq \prod_{i=1}^{m} 2^{i-1}m^{1/2}
= 2^{m(m-1)/2}m^{m/2}. 
\end{align*}
By Claim~\ref{clm:deteminant} and an easy calculation, 
we obtain $\sqrt{\,\log_2 p}<m$ in this case. 

\vspace{2mm}
(ii) Suppose $k\leq m$. 
Then $\|\mathbf{u}_{i}\|\leq 2^{\rho_m(i)-1}k^{1/2}$ for $i\in\{1,\ldots,k\}$ by Claim~\ref{clm:maximal_entry}~(3). 
By the Hadamard inequality, we have
\begin{align*}
\det U_0 
\leq \prod_{i=1}^{k}\|\mathbf{u}_{i}\|
\leq \prod_{i=1}^{k}2^{\rho_m(i)-1}k^{1/2}
\leq \prod_{i=1}^{k} 2^{m-i}k^{1/2}
= 2^{k(2m-k-1)/2}k^{k/2}. 
\end{align*}
By Claim~\ref{clm:deteminant} and an easy calculation, we have 
\[
\frac12\left( \frac{2}{k}\log_2p+k-\log_2k+1\right)\leq m. 
\]
We temporarily treat $k$ as a real variable, 
and then the minimum of the left hand side of the above inequality is attained when $\displaystyle k=\frac12\left(\log_2e+\sqrt{\left(\log_2e\right)^2+8\log_2p}\,\right)$. 
Therefore, we have
\[
\frac12\sqrt{\left(\log_2e\right)^2+8\log_2p}
-\frac12\log_2\left(\log_2e+\!\sqrt{\left(\log_2e\right)^2+8\log_2p}\,\right)
+1
\leq m. 
\]
The left hand side of the above is positive and equal to $\sqrt{2\log_2p}$ asymptotically. 
The proof is completed. 
\end{proof}
\begin{remark}
Note that we can also obtain Corollary~\ref{cor:S(L(2,1))} from the above proof. 
\end{remark}

\end{document}